\documentclass[11pt,oneside,reqno,english]{amsart}
\usepackage[T1]{fontenc}
\usepackage[latin9]{inputenc}
\usepackage{color}
\usepackage{babel}
\usepackage{amstext}
\usepackage{amsthm}
\usepackage{amssymb}
\usepackage{setspace}
\usepackage{mathtools}
\usepackage[unicode=true,pdfusetitle,
 bookmarks=true,bookmarksnumbered=false,bookmarksopen=false,
 breaklinks=true,pdfborder={0 0 0},pdfborderstyle={},backref=false,colorlinks=true]
 {hyperref}

\makeatletter
\theoremstyle{plain}
\newtheorem{thm}{\protect\theoremname}
\theoremstyle{plain}
\newtheorem{lem}[thm]{\protect\lemmaname}
\newtheorem{prop}[thm]{Proposition}

\usepackage{calrsfs}
\usepackage{graphicx}
\usepackage{color}
\usepackage{perpage}
\usepackage{upquote}
\usepackage{bbm}
\usepackage[shortlabels]{enumitem}
\usepackage{stmaryrd} 
\usepackage{extarrows} 
\usepackage{romanbar}

\MakePerPage{footnote}
\gdef\SetFigFontNFSS#1#2#3#4#5{} 
\gdef\SetFigFont#1#2#3#4#5{} 
\def\clap#1{\hbox to 0pt{\hss#1\hss}}

\DeclareMathOperator{\cov}{cov}
\DeclareMathOperator{\essmax}{ess\, max}
\DeclareMathOperator{\essmin}{ess\, min}

\def\eps{\varepsilon}

\def\PP{\mathbb{P}}

\def\RR{\mathbb{R}}
\def\cY{\mathcal{Y}}
\def\cX{\mathcal{X}}

\definecolor{myblue}{rgb}{0.09,0.32,0.44} 
\hypersetup{pdfborder={0 0 0},pdfborderstyle={},colorlinks=true,linkcolor=myblue,citecolor=myblue,urlcolor=blue}

\theoremstyle{remark}
\newtheorem*{qst*}{Question}
\newtheorem*{rmrks*}{Remarks}

\newlength{\tempindent}
\newcommand{\lazyenum}{
\setlength{\tempindent}{\parindent}
\begin{enumerate}[leftmargin=0cm,itemindent=0.7cm,labelwidth=\itemindent,labelsep=0cm,align=left,label=\arabic*)]
\setlength{\parskip}{\smallskipamount}
\setlength{\parindent}{\tempindent}
}

\makeatletter
\renewcommand{\andify}{%
  \nxandlist{\unskip, }{\unskip{} \@@and~}{\unskip{} \@@and~}}
\def\author@andify{%
  \nxandlist {\unskip ,\penalty-1 \space\ignorespaces}%
    {\unskip {} \@@and~}%
    {\unskip \penalty-2 \space \@@and~}%
}
\let\@wraptoccontribs\wraptoccontribs
\makeatother

\ifpdf
\def\afs#1#2{\href{#1}{\nolinkurl{#2}}}
\else
\def\afs#1#2{\burlalt{#1}{#2}}
\fi

\makeatother

\providecommand{\lemmaname}{Lemma}
\providecommand{\theoremname}{Theorem}

\begin{document}

\title[Arithmetic progressions in Brownian motion]{Arithmetic progressions in the trace of Brownian motion in space}
\author{ Itai Benjamini \and Gady Kozma}

\date{13.10.18 }

\begin{abstract}
It is shown that the trace of $3$ dimensional Brownian motion contains arithmetic progressions of length $5$
and no arithmetic progressions of length $6$  a.s.
\end{abstract}

\maketitle

\section{Introduction}

In this note we comment that a.s.\ the trace of a $3$ dimensional Brownian motion contains arithmetic progressions of length $5$, and no arithmetic progressions of length $6$.

Similarly, the maximal arithmetic progression in the trace of Brownian motion in $\mathbb{R}^{d}$
is $3$ for $d = 4,5$ and $2$ above that (we will only prove the three dimensional result here). On the other hand, the trace of a $2$ dimensional Brownian motion
a.s.\ contains arbitrarily long arithmetic progressions starting at the origin and having a fixed difference.

Consider $n$ steps simple random walk on the $d$ dimensional square grid $\mathbb{Z}^{d}$,
look at the number of arithmetic progressions  of length $3$ in the range, study the distribution and large deviations?

\smallskip
\noindent
{\bf Question:}
In the  large deviations regime, is there a deterministic  limiting shape? 
\smallskip



\section{Proofs}
We start with the two dimensional case. 

\begin{prop}
The trace of $2$ dimensional Brownian motion
a.s. contains arbitrarily long arithmetic progressions starting at the origin and having a fixed difference.
\end{prop}

\begin{proof}
Given a set $S$ of Hausdorff dimension $1$ in the Euclidean plane,  $2$ dimensional Brownian motion $W$ running for
unit time will   intersect  $S$ in a set of Hausdorff dimension $1$ as well, with positive probability, see e.g. \cite{P}, \cite{MP}.
Examine the unit circle. With positive probability Brownian motion run for unit time  intersects the unit circle in a set $S_1$ of dimension $1$. To each point in $S_1$ add it to itself to get $S_2$ a set of dimension $1$. Let $\tau_1\ge 1$ be the first time (after 1) our Brownian motion hits the circle with radius $3/2$. Examine it now in the time interval $[\tau_1,\tau_1+1]$.
By the Harnack principle \cite[Theorem 3.42]{MP}, the probability that Brownian motion started from $W(\tau_1)$ to intersect $S_2$ in a set of dimension 1 is comparable to that of Brownian motion starting from 0 which, as already stated, is bounded away from 0. Hence $W[\tau_1,\tau_1+1]$ will again intersect $S_2$ in a set of dimension $1$. To each point in the intersection of the form $2x, x\in S_1$ add $x$ and call the resulting set $S_3$, again of dimension 1. Continue in the same manner to get arbitrarily long arithmetic progressions. Scale invariance implies that we get arbitrarily long arithmetic progression with probability $1$.
\end{proof}

The argument above shows that with positive probability the trace of a unit time two dimensional Brownian motion
admits uncountably many arithmetic progression of arbitrary length and difference $1$.
\medskip

We now prove the high dimensional result. 

\begin{lem}\label{lem:no6}
A $3$-dimensional Brownian motion contains no arithmetic progressions
of length 6, a.s.
\end{lem}

\begin{proof}
By scaling invariant we may restrict our attention to arithmetic progressions
contained in the unit ball $B$, and to spacings at least $\delta$
for some $\delta>0$. Denote the Brownian motion by $W$. If it contains
an arithmetic progression then for every $\varepsilon>0$ one may find
$x_{1},\dotsc,x_{6}\in B\cap\frac13 \varepsilon\mathbb{Z}^{d}$ such that
$W\cap B(x_{i},\varepsilon)\ne\emptyset$ and such that the $x_{i}$
form an $\varepsilon$-approximate arithmetic progressions, by which
we mean that $|x_{i-1}+x_{i+1}-2x_{i}|\le4\varepsilon$ for $i=2,3,4,5$.
Further, the $x_{i}$ are $\delta$-separated in the sense that $|x_{i}-x_{i+1}|\ge\delta-2\varepsilon$
for $i=1,2,3,4,5$. Denote the set of such $x_{i}$ by $\mathcal{X}$
and define
\begin{gather*}
H_x=\mathbbm{1}\{W\cap B(x_{i},\varepsilon)\ne\emptyset\;\forall i\in\{1,\dotsc,6\}\}\qquad x=(x_{1},\dotsc,x_{6})\\
X=X(\varepsilon)=\sum_{x\in\mathcal{X}}\mathbbm{1}\{H_x\}.
\end{gather*}
We now claim that
\begin{equation}
\mathbb{E}(X)\le C\qquad\mathbb{E}(X^{2})\ge c|\log\varepsilon|\label{eq:moments}
\end{equation}
where the constants $c$ and $C$ may depend on $\delta$. Both calculations
are standard: the first (that of $\mathbb{E}(X)$), is an immediate
corollary of the fact that 3$d$ Brownian motion starting from 0 hits the ball $B(v,\eps)$ with probability $\approx\eps/(|v|+\eps)$, see e.g.\ \cite[corollary 3.19]{MP}. Here and below, $\approx$ means that the ratio of the two quantities is bounded above and below by constants that depend only on $\delta$. This gives
\begin{equation}\label{eq:miloyodea}
\PP(H_x)\approx \frac{\eps^6}{d(0,x)+\eps}
\end{equation}
where $d(0,x)\coloneqq\min\{d(0,x_i):i=1,\dotsc,6\}$.
Denote by $\mathcal{X}_n$ the set of $x\in\mathcal{X}$ such that $\eps 2^n<d(0,x)\le\eps 2^{n+1}$, with $\mathcal{X}_0$ having the lower bound removed. We can now write
\[
\mathbb{E}(X)=\sum_{n=0}^{\log 1/\varepsilon}\sum_{x\in\mathcal{X}_n}\mathbb{P}(H_x)
\stackrel{\textrm{\eqref{eq:miloyodea}}}{\le}
C\sum_{n=0}^{\log 1/\varepsilon}2^{3n}\cdot 2^{-n}\cdot \eps^{-3}\cdot \eps^5\le C
\]
where $2^{3n}$ is the number of possibilities for the $x_i$ closest to 0 for $x\in\mathcal{X}_n$ (this, and all other quantities in this explanation are up to constants); where $2^{-n}$ is $\PP(W\cap B(x_i,\eps)\ne\emptyset)$; where $\eps^{-3}$ is the number of possibilities for $x_2-x_1$ (we use here that the determination of $x_1$ and $x_2-x_1$ leave only a constant number of possibilities for $x_3,\dotsc,x_4$); and where $\eps^5$ is the probability to hit all of $B(x_1,\eps),\dotsc,B(x_6,\eps)$ except $B(x_i,\eps)$ given that you have hit $B(x_i,\eps)$.

The calculation of $\mathbb{E}(X^{2})$ is similar, we write $\mathbb{E}(X^{2})=\sum_{x,y\in\mathcal{X}}\mathbb{P}(H_x\cap H_y)$
and estimate the probability directly. We get about constant contribution
from each set $\{x,y:|x_i-y_i|\approx 2^{-n}\;\forall i\}$ for every $n$, hence the $|\log\varepsilon|$ term.

We now make a somewhat stronger claim on the interaction between different
$x$. We claim that there exists $\lambda>0$ such that, for any $x$,
\begin{equation}
\mathbb{P}(H_x\cap\{X\le\lambda|\log\varepsilon|\})\le\frac{C}{|\log\varepsilon|}\mathbb{P}(H_x).\label{eq:conditioned}
\end{equation}
To see this fix $x$ and let, for each scale $k\in\{1,\dotsc,\lfloor|\log\varepsilon|\rfloor\}$,
\begin{gather*}
  X_{k}\coloneqq\sum_{y\in\cY_k}\mathbbm{1}\{H_y\}\\
  \cY_k\coloneqq\{y\in\mathcal{X}:2^{k}\varepsilon\le|y_{i}-x_{i}|<2^{k+1}\varepsilon\quad\forall i\in\{1,\dotsc,6\}\}
\end{gather*}
($X_k$ depends on $x$, of course, but we omit this dependency from the notation). A calculation identical to the above shows that $\mathbb{E}(X_{k}\,|\,H_{x})\ge c$ and
$\mathbb{E}(X_{k}^{2}\,|\,H_{x})\le C$ so
\begin{equation}\label{eq:shalosh vakhetzi}
\mathbb{P}(X_{k}>0\,|\,H_{x})\ge c.
\end{equation}
Further, the events $X_{k}>0$ (still conditioned on $H_{x}$) are
approximately independent in the following sense:
\begin{lem}\label{lem:pffff}For each $x\in\cX$ and $k\in \{1,\dotsc,\lfloor\log(\delta/4 \eps)\rfloor\}$,
\begin{equation}
\cov(X_{k}>0,X_{l}>0\,|\,H_{x})\le2e^{-c|k-l|}.\label{eq:cov}
\end{equation}
\end{lem}
\begin{proof}
Assume for concreteness that $k<l$ and that $l-k$ is sufficiently large (otherwise the claim holds trivially, if only the $c$ in the exponent is taken sufficiently small). Define two radii $r<s$ between $2^k\eps$ and $2^l\eps$ as follows:
\[
r\coloneqq 2^{(2/3)k+(1/3)l}\eps\qquad s\coloneqq 2^{(1/3)k+(2/3)l}\eps.
\]
Next, define a sequence of stopping times: the even ones for exiting balls of radius $s$ and the odd ones for entering balls of radius $r$. In a formula, let $\tau_0=0$ and 
\begin{align*}
  \tau_{2m+1}&\coloneqq\inf\Big\{t\ge \tau_{2m}:W(t)\in\bigcup_{i=1}^6 B(x_i,r)\Big\}\\
  \tau_{2m}&\coloneqq\inf\Big\{t\ge \tau_{2m-1}:W(t)\not\in\bigcup_{i=1}^6 B(x_i,s)\Big\}
\end{align*}
Let $M$ be the first number such that $\tau_{2M+1}=\infty$. With probability 1 $M$ is finite. We now claim that
\begin{equation}\label{eq:Mis6}
\PP(H_x\cap\{M\ge 6+\lambda\})\le \frac{\eps^6}{d(0,x)+\eps}\big(Cr/s\big)^\lambda\qquad\forall\lambda=1,2,\dotsc
\end{equation}
To see \eqref{eq:Mis6} assume $d(0,x)>c$ for simplicity. Then every visit to $B(x_i,\eps)$ from $\partial B(x_i,r)$ ``costs'' $\eps/r$ in the probability, while every visit of $B(x_j,r)$ from $\partial B(x_i,s)$ costs $r/s$ if $i=j$ and $r$ if $i\ne j$. Since $H_x$ requires a visit to all of $x_1,\dotsc,x_6$ we have to pay the costs $\eps/r$ and $r$ at least 6 times, and the costs of $r/s$ (or $r$, which is smaller) at least $\lambda$ times. Counting over the order in which these visits happen adds no more than a $C^\lambda$. This shows \eqref{eq:Mis6} in the case that $d(0,x)>c$. The other case is identical and we skip the details.

\def\cK{\mathcal{K}}
\def\cL{\mathcal{L}}
Since $\PP(H_x)\approx \eps^6/(d(0,x)+\eps)$ (recall \eqref{eq:miloyodea}) this shows that the case $M>6$ is irrelevant. Indeed, if we define $\cK=\{X_k>0\}\cap\{M=6\}$ and $\cL=\{X_l>0\}\cap\{M=6\}$ then
\begin{equation}\label{eq:khamesh vakhetzi}
|\cov(\cK,\cL|H_x)-\cov(X_k>0,X_l>0|H_x)|\le \frac{Cr}{s}
\end{equation}
(for $l-k$ sufficiently large) and we may concentrate on $\cov(\cK,\cL|H_x)$.

Let $\mu$ be the measure on $\RR^{36}$ giving the distribution of $W(\tau_1),\dotsc,W(\tau_{12})$ (we will not distinguish between $(\RR^{3})^{12}$ and $\RR^{36}$). For an event $E$ we will use $\PP(E|W=u)$ as a short for
\[
\PP(E\,|\,W(\tau_i)=u_i\;\forall i\in\{1,\dotsc,12\},M=6)
\]
(which is of course a $\mu$-almost everywhere defined function). We next observe that for $E$ equal to any of $\cL$, $H_x$ and $\cK\cap H_x$ the function $\PP(E\,|\,W=u)$ is nearly constant i.e.
\begin{equation}\label{eq:constant}
\frac{\essmax \PP(E\,|\,W=u)}{\essmin \PP(E\,|\,W=u)} \le 1+2 e^{-c|k-l|}
\end{equation}
This is because $\cK$ and $H_x$ depend only on the behaviour inside the balls $B(x_i,2^k\eps)$ while $u_{2m+1}$ are on $\partial B(x_i,r)$. This follows from the well-known fact that the distribution of $W$ on the first hitting times (after $\tau_{2m+1}$) of $B(x_i,2^k\eps)$ is independent of $u_{2m+1}$, up to an error of $(2^k\eps)/r$; and similarly, the conditioning on exiting $B(x_i,s)$ at $u_{2m+2}$ only adds an error of $(2^k\eps)/s$. For the convenience of the reader we recall briefly how this is shown: consider Brownian motion started from a $y_1\in\partial B(x_i,r)$ and let $y_2\in \partial B(x_i ,2^{k+1}\eps)$ be the first point visited in $B(x_i,2 ^{k+1}\eps)$, let $y_3$ be the last, and let $y_4$ be the first point visited in $B(x_i,s)$. Then the joint distribution of $y_2$, $y_3$ and $y_4$ can be written easily using the Poisson kernel (see \cite[Theorem 3.44]{MP} for its formula). For example, the density of $y_2$ is $(r-\eps 2^{k+1})|y_2-y_1|^{-3}$ (the density in $\RR^3$) from which we need to subtract the density after exiting $B(x_i,s)$, which is given by an integral of similar expressions. The exact form does not matter, only the fact that the $y_1$ dependency comes from the term $|y_2-y_1|$ is nearly constant in $y_1$ in the sense above. The same holds for the density of the transition from $y_3$ to $y_4$ and the density between $y_2$ and $y_3$ is of course completely independent of $y_1$ and $y_4$. Conditioning on exiting in a given $y_4$ is merely restricting to a subspace and normalising, conserving the near independence. This justifies (\ref{eq:constant}) in this case.

We have ignored here the case that $0\in B(x_i,r)$ for some $i$, in which case $u_1$ is inside $B(x_i,r)$ rather than on its boundary, but in this case $u_1$ is constant and certainly does not affect anything. This shows \eqref{eq:constant} for $E=H_x$ and $\cK\cap H_x$.

The argument for the other case is similar, becuase $\cL$ depends only on what happens outside $B(x_i,2^l\eps)$ and $u_{2m}$ is on $\partial B(x_i,s)$ (this time without exceptions). Hence we have only an error of $s/(2^l\eps)$. All these errors are exponential in $l-k$. This shows \eqref{eq:constant} is all 3 cases. In particular we get, for all three cases for which \eqref{eq:constant} holds, that
\begin{equation}\label{eq:constant2}
  \PP(E\,|\,W=u)=\PP(E)(1+O(e^{-c|k-l|}))
\end{equation}
which holds for $\mu$-almost every $u$.

The last point to note is that, conditioning on $W=u$ makes $\cL$ independent of $H_x$ and of $\cK$ as the first depends only on what happens in the odd time intervals, i.e.\ between $\tau_{2m}$ and $\tau_{2m+1}$, $m=0,\dotsc,6$ while the other two depend on what happens in the even time intervals, between $\tau_{2m-1}$ and $\tau_{2m}$, $m=1,\dotsc,6$. Hence
\begin{align*}
  \PP(\cL\cap H_x)&=\int \PP(\cL\cap H_x\,|\,W=u)\,d\mu(u)\\
  \textrm{by independence}\qquad &=\int\PP(\cL\,|\,W=u)\PP(H_x\,|\,W=u)\,d\mu(u)\\
  \textrm{by \eqref{eq:constant2}}\qquad &=
  \PP(\cL)\PP(H_x)(1+O(e^{-c|k-l|})).
\end{align*}
A similar argument gives
\[
\PP(\cL\cap\cK\cap H_x)=\PP(\cL)\PP(\cK\cap H_x)(1+O(e^{-c|k-l|})).
\]
Together these two inequalities bound $\cov(\cK,\cL\,|\,H_x)$. With \eqref{eq:khamesh vakhetzi} the lemma is proved.
\end{proof}

With (\ref{eq:cov}) established we can easily see (\ref{eq:conditioned}), by
using Chebyshev's inequality for the variable
$\#\{k:X_{k}>0\}$, with \eqref{eq:shalosh vakhetzi} giving the first moment and \eqref{eq:cov} the covariance. (In fact, it is not difficult to get a much better
estimate than $C/|\log\varepsilon|$, an $\varepsilon^{c}$ is also
possible. But we will not need it).

Summing (\ref{eq:conditioned}) over all $x$ and using \eqref{eq:moments} gives
\[
\mathbb{P}(X\in(0,\lambda|\log\varepsilon|))\le\frac{C}{|\log\varepsilon|}
\]
This, with $\mathbb{E}(X)\le C$ shows that $\mathbb{P}(X>0)\le C/|\log\varepsilon|$,
proving lemma \ref{lem:no6}.
\end{proof}
\begin{lem}
A $3$-dimensional Brownian motion contains arithmetic progressions
of length 5, a.s.
\end{lem}

\begin{proof}
Let $\varepsilon$ and $X=X(\varepsilon)$ be as in the proof of the
previous lemma (except we now fix $\delta$ to be, say, $\frac{1}{10}$).
It is straightforward to calculate
\[
\mathbb{E}(X(\varepsilon))\ge\frac{c}{\varepsilon}\qquad\mathbb{E}(X(\varepsilon)^{2})\le\frac{C}{\varepsilon^{2}}
\]
 which show that $\mathbb{P}(X(\varepsilon)>0)\ge c$. A simple calculation
shows that for some $\lambda>0$ we have that $X(\lambda\varepsilon)>0\implies X(\varepsilon)>0$.
Hence $\{X(\lambda^{k})>0\}$ is a sequence of decreasing events with
probabilities bounded below. This implies that
\[
\mathbb{P}\Big(\bigcap_{k}\{X(\lambda^{k})>0\}\Big)>0.
\]
The event of the intersection can be described in words as follows:
for every $k$ there exists $x_{1}^{(k)},\dotsc,x_{5}^{(k)}\in B$
which are $\frac{1}{10}$-separated and $\lambda^{k}$-approximate
arithmetic progression such that $W\cap B(x_{i}^{(k)},\lambda^{k})\ne\emptyset$
for $i\in\{1,\dotsc,5\}$. Taking a subsequential limit we get $x_{i}^{(k_{n})}\to x_{i}$
and these $x_{i}$ will be $\frac{1}{10}$-separated, will form an
arithmetic progression, and will be on the path of $W$. So we conclude
\[
\mathbb{P}(W\text{ contains a 5-term arithmetic progression in }B)>0.
\]
Scaling invariance now shows that the probability is in fact $1$.
\end{proof}

\end{document}